\documentclass[11pt]{amsart}
\usepackage{amssymb, latexsym, amsmath, amscd, array, graphicx, enumerate, ,tikz-cd}

\usepackage{mathrsfs}%For mathscr

\swapnumbers
\numberwithin{equation}{section}

\def\m{\medskip}

\newtheorem{thm}{Theorem}[section]

\newtheorem{proposition}[thm]{Proposition}
\newtheorem{corollary}[thm]{Corollary}
\newtheorem{question}[thm]{Question}

\newtheorem{prop}[thm]{Proposition}

\newtheorem{cor}[thm]{Corollary}

\theoremstyle{definition}
\newtheorem{definition}[thm]{Definition}

\newtheorem{rem}[thm]{Remark}

\newtheorem{df}[thm]{Definition}
\newtheorem{example}[thm]{Example}
\newcommand\theoref{Theorem~\ref}

\newcommand\propref{Prop.~\ref}
\newcommand\proprefs{Propositions~\ref}
\newcommand\corref{Corollary~\ref}
\newcommand\defref{Definition~\ref}

\newcommand\remref{Remark~\ref}

\newcommand{\CC}{\mathbb{C}}

\newcommand{\QQ}{\mathbb{Q}}
\newcommand{\RR}{\mathbb{R}}
\newcommand{\ZZ}{\mathbb{Z}}

\newcommand{\ga}{\alpha}
\newcommand{\gb}{\beta}

\newcommand{\gf}{\varphi}

\def\wgt{\operatorname{wgt}}
\def\TC{\operatorname{TC}}
\def\TCP{\operatorname{TC^{Pav}}}
\def\TCRD{\operatorname{TC^{RD}}}
\def\TCMW{\operatorname{TC^{MW}}}
\def\TCS{\operatorname{TC^{S}}}
\def\pt{\operatorname{pt}}

\def\cat{\operatorname{cat}}
\def\relcat{\operatorname{relcat}}

\def\secat{\operatorname{secat}}
\def\sec{\operatorname{secat}}
\def\id{\operatorname{id}}
\def\zcl{\operatorname{zcl}}
\def\ts{\times}

\def\ov{\overline}
\def\cl{\operatorname{cl}}

\long\def\forget#1\forgotten{} %

\begin{document}
\title[]{Maps of degree one, relative LS category and higher topological complexities}
%\title{On LS-category and topological complexity for maps of degree one}
%[Higher topological complexity of a map]

\author[Y. B. Rudyak]{Yuli B. Rudyak}
\address{Department of Mathematics, University of Florida, Gainesville, FL 32611-8105}
\email{rudyak@ufl.edu}

\author[S. Sarkar]{Soumen Sarkar}
\address{Department of Mathematical, IIT Madras, Chennai 600036, India}
\email{soumen@iitm.ac.in}

\date{\today}
\thanks{
}

\maketitle 
\abstract 

In this paper, we introduce relative LS category of a map and study some of its properties. Then we introduce `higher topological complexity' of a map, a homotopy invariant. We give a cohomological lower bound and compare it with previously known `topological complexity' of a map.  Moreover, we study the relation between Lusternik-Schnirelmann category and topological complexity of two closed oriented manifolds connected by a degree one map. 
\endabstract

\section{Introduction}
The concept of ``sectional category'' of  a fibration $p \colon E \to B$  was introduced by Schwarz in~\cite{sv} under the name ``genus''. James \cite{J} proposes to replace the overworked term ``genus'' by ``sectional category''.  This invariant $\secat(p)$ is one less than the Schwarz genus of $p$ in~\cite{sv}. In particular, $\secat(p)=0$ if and only if $p$ has a section. Sectional category of $p$ is the minimum cardinality of the open coverings of $B$ such that on each open set in the covering there is a homotopy section of $p$. If no such integer exists then it is convention that $\secat(p) = \infty$.   If $E$ is contractible and $p$ is surjective then $\secat(p)$ is equal to the classical   Lusternik--Schnirelmann category (in short LS category throughout this paper) of  $B$. Several applications and properties of sectional category and LS-category can be found in \cite{CLOT}. 

\vspace{0.1 cm}  

There is a generalization of LS category called category of a map $f \colon X \to Y$. We recall that LS category of $f$ is the minimum number $\cat(f)$ such that $X$ can be covered by $\cat(f) + 1$  many open subsets and the restriction of $f$ on each of these  open subsets is null-homotopic.   In particular, $\cat(i) = \cat(X)$.  It is a homotopy invariant and 
  \begin{equation}\label{eq_prod} 
\cat(f \times g) \leq \cat(f) + \cat(g)
\end{equation} see \cite{CLOT, Sta}. In this paper,  we, in particular introduce relative LS category of a map on a pair of spaces  and study some of its properties. 

\vspace{0.1 cm}  
Another well known particular case of sectional category is Farber's topological complexity \cite{Far}. He introduced topological complexity of a configuration space  to understand the navigational complexity of the motion of a robot. We denote  the closed unit interval $[0,1]$ by $I$. Let $X$ be a path connected Hausdorff space and $X^{[0,1]}$ the free path space in $X$ equipped with compact-open topology and 
\[
\pi \colon X^I  \to X \times X
\]
 the free path fibration defined by $\pi(\alpha) = (\alpha(0), \alpha(1))$ for $\alpha\in X^I$. Then  $\secat(\pi)$ is equal to Farber's topological complexity, denoted by $\TC (X)$. 

\m So, we have close relatives: topologisal complexity and Lusternik--Schnirelmann category. Both are numerical invariants and are special cases of Schwarz genus (sectional category).  Thus, since we work with  the pair (LS category of spaces, LS category of maps) and in view of parallelism between LS and TC, it seems reasonable to  loop the presentation and introduce TC of mappings.

A. Dranishnikov  enquired an appropriate definition for topological complexity of a map in his talk at the CIEM, Castro Urdiales in 2014. Let $f \colon X \to Y$ be a map. Given $y_1, y_2 \in \mbox{Im}{f}$, find a continuous motion planner $\alpha$ in $X$ such that $f(\alpha(0))=y_1$ and $f(\alpha(1))=y_2$.  To answer this, we introduce a definition of ``(higher) topological complexity'' of a map.

\forget
 which is a new homotopy invariant.

{\color{red} Maybe to remove ``which is a new topology invariant''?}
%\forgotten

 \vspace{0.1 cm}  
We note that there is a notion of topological complexity of a map, but it lacks topologist's basic demand namely homotopy invariance.

{\color{red}I suggest to replace the above sentence as follow: We note that there are several notions of topological complexity of a map, but most of them lacks topologist's basic demand namely homotopy invariance. }
\forgotten

 \vspace{0.1 cm}  
 We note that there are several notions of topological complexity of a map, but most of them lacks topologist's basic demand namely homotopy invariance. The basic goal to define map from one state $X$ to another state $Y$  is to study properties of one from another. For example, if $f \colon X \to Y$ and $f$  is nice enough then the complexity of the state $X$ could be approximate by the complexity of the state $Y$, and this expectation is very natural. On the other hand, topologists expect that if $X$ and $Y$ are less complex then complexity of $f$ has be less. In particular, if $X$ and $Y$ are contractible, that is there is no obstruction in both the state, then from topological point of view, complexity of a map from $X$ to $Y$ should be trivial. But an Example in page 4 of \cite {Pav} contradicts this topological expectation. So, we believe our definition is more appropriate according to topological point of view.

\vspace{0.1 cm}  
Recently Murillo and Wu~\cite{MuWu} gave a definition of `topological complexity' of a (work) map which is a homotopy invariant. They gave a cohomological lower bound for this invariant. However, our approach is totally different to define `topological complexity, of a map as well as we  introduce higher topological complexity of a map.   

 \vspace{0.1 cm}  
 The paper is organize as follows. 
 In Section \ref{sec_rel_cat_weight}, we recall two versions of relative category and introduce the concept of relative LS category of a map which is denoted by $\relcat(f)$. We also study several properties of this invariant which generalizes some of the properties of category of a map. We show that
$$  \relcat(f \times g) \leq \relcat(f) + \relcat(g)$$
  %\end{equation}
and give a cohomological lower bound of $\relcat(f)$ in Proposition \ref{prop_prod_ineq} and Theorem \ref{thm:cohom_low_bnd} respectively. We ask the Question \ref{gen_Ganea_conj} following Ganea conjecture.  

\vspace{0.1 cm}  
In Section \ref{sec_top_comp_map}, we introduce the concept of `higher topological complexity', denoted by $\TC_n(f)$,  of a map $f \colon X \to Y$. We show that it is a homotopy invariant and the following holds.
 $$\TC_n(f) \leq \min \{\TC_n(X), \TC_n(Y)\}$$ 
and  $$\cat(f) \leq \TC_n(f) \leq \TC_{n+1}(f) \leq \cat(f^{n+1}) \leq (n+1) \cat(f)$$ for all  $n \geq 2$.
That is growth of $\TC_n(f)$ is linear. 
The works of \cite{Pav, RaDe} and \cite{MuWu} study `topological complexity' of a map. However, we discuss higher analogue of it. We compare our definition with the previous ones. 

\vspace{0.1 cm}  
Section \ref{sec_lscat_deg_map} is dedicated to answer the following open problem for a broader class of manifolds. This question was asked by Rudyak \cite{Rud99} and now it is consider as a conjecture, see \cite{Dra, DS}. 

\begin{question}\label{Rud_conj}
Given two closed connected orientable manifolds $M$ and $N$, assume that there exists a map
$f \colon M \to N$ of degree $\pm 1$. Is it true that $\cat (M) \geq \cat  N $?
\end{question} 

\begin{rem}\label{r:deg}
 Since we speak about the degree of a map $M\to N$, the manifolds $M$ and $N$ (like in \ref{Rud_conj}, say) must be assumed to be oriented. So, in future we will not mention ``oriented'' in such situations
 \end{rem}

\vspace{0.1 cm}  
We replace $\cat$ by $\TC$ in Question \ref{Rud_conj} and ask a similar question in Section \ref{sec_topcom_deg_map}. We give affirmative answer to this question too for several category of manifolds and when the dimensions are less than 4.  

\m Through the paper, the word ``space'' means ``topological space'' unless some other is said explicitly, and all maps of spaces are assumed to be continuous. The word `smooth'' means $C^{\infty}$ for spaces as well as for maps. Finally, all manifolds are assumed to be smooth and connected unless some other is said explicitly.

\section{Two Versions of Relative Category} \label{sec_rel_cat_weight}
In this section, we introduce relative category of a map and study some of its properties. 
Let $\cat X$ denote the LS category of a space $X$. Note that we use the normalized category, that is $\cat$ of the point is zero. Many topological constructions exploit relative version, informally saying to use an object ``modulo subspaces''. Here we use two following definitions. 

\begin{definition}[{\cite[Definition 7.1]{CLOT}}]\label{d:clot}
Let $(X, A)$ be a pair of topological spaces. The relative category $\cat(X, A)$ is the least non-negative integer $k$ such that $X$ can be cover by open sets $V_0, V_1, \ldots, V_k$ with $A\subseteq V_0$  and such that, for $i\geq 1$, the sets $V_i$ are contractible in $X$, and there exists a homotopy of pairs $H \colon (V_0 \times [0,1], A \times [0, 1]) \to (X, A)$ with $H(-, 1)$ is given by the inclusion map $V_0 \hookrightarrow X$ and $H(V_0, 0) \subseteq A$.
\end{definition}

\begin{definition}[{\cite[Definition 2.1, 2.2]{LuMa}}]\label{d:relcat}
\
\begin{enumerate}
\item Let $(X, A) $ be a pair of topological spaces. An open subset $U$ of $X$ is called \emph{relative categorical} if there is a homotopy $H \colon U \times [0,1] \to X$ such that $H{(-, 1)}$ is the inclusion $U \hookrightarrow X$ and $H(U, 0) \subseteq A$.  
\item Relative category of the pair $(X, A)$, denoted by $\relcat(X,A)$ is the least integer $n$ such that $X$ can be cover by $n+1$ many relative categorical sets. Otherwise, we say this is infinity.
\end{enumerate}
\end{definition}

Note that $\cat(X, x_0)=\cat X=\relcat (X, x_0)$ where $x_0 \in X$. The following example shows that  
\defref{d:clot} and \defref{d:relcat} are not equivalent, in general. 

\begin{example}
Let $X=D^n$ be the unit closed $n$-ball in $\mathbb{R}^n$ and $A=\partial{D^n}$ its boundary. We prove that  $\relcat(X, A)=0$ and $\cat(X, A)=1$.

\m First, note that $\cat(X, A) \geq 1$, since the open subset $V_0=X$ of $X$ does not satisfy the homotopy condition in Definition \ref{d:clot}. Now, let $V_0= X - \{0\}$ be he punctured $n$-ball and $V_1 = X -A$. Then $\{V_0, V_1\}$ is an open covering of $X$. Observe that this covering satisfy the conditions in Definition \ref{d:clot}. Therefore, $\cat(X, A)=1$.

On the other hand, the open set $V_0=X$ is contractile to a point in $A$, and thus $\relcat(X,A)=0$. \qed
\end{example}

Now we are in a position to introduce the concept of relative category of a map. Then we study their several properties in the remaining of this section.  

\begin{definition}\label{def_rel_inessential}
Let $g \colon (B, C) \to (X, A)$ be a map of pairs ($C$ could be empty subset). Then $g$ is called \emph{relatively inessential} map if there is a homotopy $H \colon B \times [0, 1] \to X$ such that  $H(-, 1)$ is $g$, $H(B, 0) \subseteq A$, $H(C, t) \subseteq A$ for all $t \in [0,1]$.
\end{definition}
 
\begin{definition}\label{d:relcat_map}
Let $f \colon (B, C) \to (X, A)$ be a map of pairs. Then $ \relcat(f) \leq \ell$ if $B$ has an open cover by $B_0, \ldots, B_{\ell}$ such that $f|_{(B_i, B_i\cap C)}$ is relatively inessential. Minimum $\ell$ with this property is called relative-category of $f$. Otherwise we say it is infinity.
\end{definition}

Clearly, if $C= \emptyset $ or $C = \{pt\}$ and  $ A = \{pt\}$ then $\relcat(f)$ turns into known invariant $\cat(f)$ for $f \colon B \to X$, see  \cite[Exercise 1.16]{CLOT}.

\begin{thm}\label{thm_hom_inv}
The number $ \relcat(f)$ is a homotopy invariant.
\end{thm}

\begin{proof}
Let $f, g \colon (B, C) \to (X, A)$ be homotopic maps relative to $C$. Then there is a homotopy
$
H \colon B \times [0, 1] \to X
$
 such that $H(b, 0) = f(b)$, $H(b, 1)=g(b)$ for all $b \in B$,
and $H(C, t) \subseteq A$ for all $t \in [0, 1]$. Let $B_i$ be an open subset of $B$ such that the map
\[
g|_{(B_i, B_i \cap C)} \colon (B_i, B_i\cap C) \to (X, A)
\]
 is relatively inessential. Now we use the homotopy $H$ to concude that  
 \[
 f|_{(B_i, B_i \cap C)} \colon (B_i, B_i\cap C) \to (X, A)
 \]
 is relatively inessential. Therefore, $ \relcat(f) \leq ~ \relcat(g)$. Similarly, the opposite inequality holds.
\end{proof}

Note that for identity map, $\id \colon (X, A) \to (X, A)$, $ \relcat(\id) =  \relcat(X, A)$. Moreover, we have the following from Definition \ref{d:relcat} and \ref{d:relcat_map}.

\begin{proposition}\label{prop_inequality_1}
Let $f \colon (B, C) \to (X, A)$ and $g \colon (Y, D) \to (B, C)$ be two maps. Then 
\begin{enumerate}[\rm (i)]
\item $\relcat(f) \leq \min\{\relcat(B, C), ~~ \relcat(X, A)\} $.
\item $\relcat(f \circ g) \leq \min\{\relcat(f), ~~ \relcat(g)\} $.
\end{enumerate}
\end{proposition}

\begin{proposition}\label{prop_prod_ineq}
Let $f \colon (B, C) \to (X, A)$ and $g \colon (E, F) \to (Y, D)$ be two maps. Then 
$\relcat(f \times g) \leq \relcat(f) + \relcat(g) $.
\end{proposition}
\begin{proof}
The proof is similar to the proof of $\cat(X \times Y) \leq \cat(X) + \cat(Y)$ \cite[Theorem 1.37]{CLOT} with some modification.
\end{proof}

We can ask the following question generalizing Ganea conjecture \cite{Gan2}.

\begin{question}\label{gen_Ganea_conj}
Let $\pt$ be a point in $S^n$. What are the pair of spaces $(X, Y)$ so that $$\relcat(X \times S^n, Y \times \{pt\}) = \relcat(X, Y) + 1?$$
\end{question}
We note that when $Y$ is a point in $X$, then Question \ref{gen_Ganea_conj} is known as Ganea conjecture, and counter example of this conjecture was first given by Iwase \cite{Iwa98}. 

\forget
\begin{thm}
 If $X$ is a CW-complex and $Y$ is a subcomplex of $X$, then $\relcat(X \times S^n, Y \times \{p\}) = \relcat(X, Y) + 1$.
\end{thm}
\begin{proof}
I am thinking about it.
\end{proof}
\forgotten

\m In the rest of this section we give a cohomological lower bound of $\mbox{relcat}(f)$. If $R$ be a commutative ring, then the {\it nilpotency index} of $R$ is the non-negative integer $n$ such that $R^n \neq 0$ but $R^{n+1} =0$ and it is denoted by $\mbox{nil} R$.  Let $f \colon (B, C) \to (X,A)$ be a map of pairs, and let $f^* \colon H^*(X, A) \to H^*(B, C)$ be the induced homomorphism. 

\begin{thm}\label{thm:cohom_low_bnd}
${\rm nil} {\rm Im }(f^*) \leq \relcat(f)$.
\end{thm}
\begin{proof}
%This is a standard argument. 
Let $\relcat(f) =k$ and  $g_i \colon (B_i, B_i \cap C) \to (X, A)$ relatively inessential map for $i=0, 1, \ldots, k$ such that $g_i$ is a restriction of $f$ and $B_0, B_1, \ldots, B_k$ covers $B$. For the triple $(B, B_i, B_i \cap C)$, we have the following long exact sequence
\[
\cdots \to H^*(B, B_i) \xrightarrow{q_i^*} H^*(B, B_i \cap C) \xrightarrow{\iota_i^*} H^*(B_i, B_i \cap C)  \to \cdots .
\]
 Also we have the following commutative diagrams induced from natural inclusions and some restrictions of $f$.
\[
\begin{tikzcd}[column sep=1.5em]
 & H^*(X, A) \arrow{dr}{g_i^*} \arrow[swap]{dl}{\bar{f}^*_i} \\
H^*(B, B_i\cap C) \arrow{rr}{\iota_i^*} &&   H^*(B_i, B_i \cap C) 
\end{tikzcd}
\]
and 
\[
\begin{tikzcd}[column sep=1.5em]
 & H^*(X, A) \arrow{dr}{\bar{f}_i^*} \arrow[swap]{dl}{f^*} \\
H^*(B, C) \arrow{rr}{\bar{\iota}_i^*} &&   H^*(B, B_i \cap C). 
\end{tikzcd}
\]
 
Suppose $\beta_0, \beta_1, \ldots, \beta_k$ belong to ${\rm Im}(f)$. Then $\beta_i = f^*(\alpha_i)$ for some $\alpha_i \in H^*(X, A)$ and $i \in \{0, \ldots, k\}$. Since $g_i$ is inessential, then $g^*$ is trivial. So $\iota^*_i(\bar{f}^*_i(\alpha_i))=0 \in H^*(B_i, B_i \cap C) $.  Thus $\bar{f}^*_i(\alpha_i)= q_i^*(\gamma_i)$ for some $\gamma_i \in H^*(B, B_i)$ and $i \in \{0, \ldots, k\}$. By relative cohomology product rule, we get $\gamma_0 \cup \gamma_1 \cup \cdots \cup \gamma_k \in H^*(B, B) =0$, and $\bar{f}^*_0(\alpha_0) \cup \bar{f}^*_1(\alpha_1) \cup \ldots \cup \bar{f}^*_k(\alpha_k) \in H^*(B, C)$. So using \cite[Proposition 3.19]{Pav} and the map $q \colon (B, C) \to (B, B)$ we can get
\begin{align*}
\beta_0 \cup \beta_1 \cup \cdots \cup \beta_k & =  f^*(\alpha_0) \cup f^*(\alpha_1) \cup \cdots \cup f^*(\alpha_k) \\ 
& =  \bar{f}^*(\alpha_0) \cup \bar{f}^*(\alpha_1) \cup \cdots \cup \bar{f}^*(\alpha_k)  \\ 
& =  q_0^*(\gamma_0) \cup q_1^*(\gamma_1) \cup \cdots \cup q_k^*(\gamma_k)\\ 
& = q^*(\gamma_0 \cup \gamma_1 \cup \cdots \cup \gamma_k) \\
&= 0.\nonumber
\end{align*}
 This proves the conclusion. 
\end{proof}

%======================================================

\section{Topological complexity of a map}\label{sec_top_comp_map}
In different context `topological complexity of a map' has been studied in  the work of \cite{Pav}, \cite{ RaDe} and \cite{MuWu}.  In this section, we introduce the concept of `higher topological complexity' of a map in a different way and show that it is a homotopy invariant. Then we compare it with the previous ones and study some properties of this new invariant.  Consider the free path  fibration 
\begin{equation}\label{pi}
\pi  \colon X^I \to X \times X, \quad \pi(\ga)=(\ga(0), \ga(1)), \, \ga: I\to X.
\end{equation}

 Fatber~\cite{Far} defined the topological complexity $\TC(X)$ of a space $X$ as the sectional category of $\pi$, and showed a nice application of this notion to robot motion planning,~\cite{Far2}. Later Rudyak~\cite{Rud10}, see also~\cite{BGRT} introduced  the “higher analogues” of topological complexity (also related to robotics, by the way). Let us recall the definition. 
 
\m Given a space $X$, consider the fibration
\begin{equation}\label{pin}
\begin{aligned}
\pi_n: X^{I}&\to X^n,\\
\pi_n(\ga)=&\left(\ga(0), \ga\left(\frac{1}{n-1}\right), \ldots, \ga\left(\frac{n-2}{n-1}\right), \ga(1)\right)
\end{aligned}
\end{equation}
where $\ga\in X^I$. 
 
 \begin{df}\label{def_higher_tcn}
A {\it higher}, or {\it sequential  topological complexity} of
order $n$ of a space $X$ (denoted by $\TC_n(X)$) is the sectional category of
$\pi _n$. So, $\TC_n(X) = \secat(\pi_n)$.
\end{df}
 
Note that $\TC_2$ coincides with the invariant $\TC$ introduces by Farber.

\m Recall that the LS category of a map $f$ have two important properties: $\cat(f)$ is a homotopy invariant of $f$, and $\cat(\id_X)=\cat(X)$.

Since $\cat$ is a closed relative of $\TC$, it seems reasonable and useful to introduce topological complexity $\TC(f)$ of a map $f$ having, in particular, properties that are similar to two above mentioned ones. 

Let $\Delta_n \colon X \to X^n$ be the diagonal of the product $X^n$ for $n \geq 2$.

\begin{definition}\label{def:TC_map}
Let $X$ and $Y$ be path connected spaces, and $f \colon X \to Y$ be a map. Let $\phi_n := f \times \cdots \times f$ ($n$-times) and 
\[
\ov\phi_n \colon (X^n, \Delta_n (X)) \to (Y^n, \Delta_n(Y))
\]
be the induced maps of pairs. 
We define the $n$-th (higher) topological complexity of $f$, denoted by $\TC_n(f)$, as $\TC_n(f) :=\relcat\ov\phi_n$.
\end{definition}  

 Now one can see thet  $\TC_n$ turns out to be a functor on the category of topological spaces.

\m Previously we had three following version of `topological complexity' of a map: the first one that is given by Pave\v{s}i\'{c}, \cite{Pav} , we denote it by $\TCP$, the second one that is given by Rami and Derfoufi, \cite{RaDe}, we denote it by $\TCRD$, and third one is by Murillo and Wu \cite{MuWu}, we denoted by $\TCMW$. Here we have the equalities 
\[
\TCP(\id_X)=\TC(X)=\TC(\id_X)=\TCRD(\id_X)== \TCMW(\id_X).
\]
 But neither $\TCP$ nor $\TCRD$ are homotopy invariant, see Example \ref{example_not_homotopy_pav} and \ref{example_not_homotopy_rade}. Next we compare them with our definition of $\TC_2(f)$.

\forget \m Next we compare our definition of $\TC_2(f)$ with two previous notions known as ``topological complexity of a map" of $\TCP$~\cite[Section 2]{Pav} and $\TCRD$~\cite[Definition 1]{RaDe}.
\forgotten

\m First we recall the definition of topological complexity $\TCP(f)$ of a map $f$. Let $q \colon E \to B$ be a surjective map. The sectional number, denoted by $\mbox{sec}(q)$, of $q$ is the smallest length $n$ of the filtrations $$\emptyset = V_0 \subset V_1 \subset \cdots \subset V_n=B$$ such that there is a section of $q^{-1}(V_i - V_{i-1}) \to V_i - V_{i-1}$ for $i=1, \ldots, n$. If there is no such integer then $\mbox{sec}(p)=\infty$. We note that sectional number and sectional category of $q$ are equal if $q$ is a fibration. Let $X, Y$ be path connected spaces and  $f \colon X \to Y$ a surjective map.  Consider the fibration 
$
\pi  \colon X^I \to X \times X$ as in \eqref{pi}.
It induces a continuous map $\pi_f \colon X^I \to X \times Y$ by $\pi_f = (\id \times f ) \circ \pi$. Now, the topological complexity $\TCP(f)$ of $f$ is defined as  the sectional number of $\pi_f$, that is 
\begin{equation}
\label{def_tc_pav}
\TCP(f) := \mbox{sec}(\pi_f).
\end{equation}
We note that \cite[Corollary 3.9]{Pav} says that $\TCP$ is a fiber homotopy invariant.

\begin{example}\label{example_not_homotopy_pav}
Let $f, g \colon [0, 3] \to [0,2]$ be two continuous functions defined by the following.
\begin{align*}
f (x) = 
     \begin{cases}
       x &\quad\text{if}~ 0 \leq x \leq 1 \\
       1 & \quad \text{if}~1 \leq x \leq 2 \\
       x-1 & \quad\quad \text{if}~ 2 \leq x \leq 3
     \end{cases}
\end{align*}
and 
\begin{align*}
g (x) = \frac{2x}{3}  \quad\quad \text{if}~ 0 \leq x \leq 3.
     \end{align*}
     
\m Hence $f$ has no continuous section, and $g$ has a section. So  $\TCP(f) > 0$ and $\TCP(g)=0$. But $f$ and $g$ are homotopic by $t f +(1-t)g, t \in [0, 1]$.
Therefore $\TCP$ is not a homotopy invariant. This is mentioned in \cite{Pav}, but an explicit example was not given. 
\end{example}

\m Now we recall the topological complexity $\TCRD(g)$ of a map  $g$. Let $Z$ be a path connected space and $g \colon Z \to W$ be a map. Then the fiber space $Z \times_{W} Z :  = (g \times g)^{-1}(\Delta W)$ is a subset of $Z \times Z$. Let 
\[
\pi^g \colon \pi^{-1}(Z \times_{W} Z) \to Z \times_W Z.
\] 
be the fibration induced from $\pi: Z^I \to Z\ts Z$  by the inclusion 

\[
Z \times_{W} Z \to Z\ts Z.
\]

Then $\TCRD(g)$ is the sectional category of $\pi^g$, that is

\begin{equation}\label{def_tc_rade}
\TCRD(g) := \sec(\pi^g).
\end{equation}
 
\m It turns out that $\TCRD$ is also a fiber homotopy invariant \cite[Corollary 7]{RaDe}. We also  note that this definition is a particular case of ``relative topological complexity" studied in \cite{Far2}. 

\begin{example}\label{example_not_homotopy_rade}
In this example we show that $\TCRD$, ``topological complexity of a map" defined in \cite{RaDe} is not a homotopy invariant. Let $X$ be a path connected topological space and $CX$ the cone on $X$ with apex $a$. Write $CX=X\ts[0,1]/X\ts\{1\}$. Define $\iota, \mathfrak{c} \colon X \to CX$ where $\iota$ is the inclusion $\iota(x)=(x,0)$ and  $\mathfrak{c}(X)=a$. Then  $\TCRD(\mathfrak{c}) = \TC(X)$ (Farber's topological complexity). 

For the map $\iota$, we have $X \times_{CX} X = \Delta X$. Then $\pi^{-1}(\Delta X)$ is the free loop space on $X$. So constant maps induce a section of $\pi^{\iota}$. Therefore $\TCRD(\iota) =1 \neq \TC(X)=\TCRD (\mathfrak c)$ in general.  It remain to note that $\iota$ and $\mathfrak{c}$ are homotopic.
\end{example}

\m In contrast to both definitions $\TCP$ and $\TCRD$ of the ``topological complexity of a map", it is clear that our definition is a homotopy invariant, a topologist primary interest. 

Now we show that $\TCP\neq \TC\neq \TCRD$ explicitly. Let $X$ be a contractible space and $Y$ a non-contractible space. Let $f \colon X  \to Y$ be a surjective fibration, for example $\exp \colon [0, 1] \to S^1$. Then by \cite[Proposition 3.2]{Pav} ${\TCP(f)} \geq \cat(Y) \geq 1$.  But our definition gives $\TC(f) = 0$. Therefore these two are different.

\vspace{0.1 cm}   On the other hand, if $g$ is a constant map then $\TCRD(g)$ is the topological complexity $\TC(X)$ of $X$ which is strictly greater than $\cat(X)$ in general.

\m Now we recall the definition of topological complexity $\TCMW(f)$ of (work) map.  Let $f \colon X \to Y$ be a continuous map. Then $\TCMW(f)$ is the least integer $n \leq \infty$ such that there exist open subsets $U_0, \ldots, U_n$ of $X \times X$ on each of which there is a map $s_i \colon U_i \to X$ satisfying $(f\times f) \circ \Delta_2 \circ s_i \simeq f \times f|_{U_i}$ for $i=0, \ldots, n$. We also note that Scott  defines topological complexity $\TCS(f)$ of a map in \cite[Definition 3.1]{Sco}, and he shows that these two definitions are equivalent. 

Let $(a,b) \in U_i$. So $s_i(a,b) \in X$. Then $(f\times f) \circ \Delta_2 \circ s_i (a, b) \in \Delta_2(Y)$.  Yet, we cannot conclude that $ f \times f|_{U_i} \colon (U, U \cap \Delta_2(X))\to (Y, \Delta_2(Y))$ is inessential in the sense of Definition \ref{def_rel_inessential}. However, if we assume  $ f \times f|_{U_i} $ is inessential, then it is homotopict to a map $U_i \to \Delta_2(Y)$. Thus by \cite[Theorem 3.4]{Sco} we get the following.   
 \begin{equation}\label{TC_RS_MW}
 \TCMW(f) = \TCS(f) \leq  \TC_2(f).
 \end{equation}
 %{\color{red}
 We do not know yet if $\TC_2(f)=\TCMW(f)$.
% Therefore, $\TC_2(f)$ gives a reasonably tight upper bound for $\TCMW(f)$. 
  Nevertheless, we give higher analogue of the topological complexity of a map which generalizes higher topological complexity of a space.

\m In the remaining we study some properties, lower bound and upper bound of $\TC_n(f)$. %  our {\it higher topological complexity of a map}.

\begin{proposition}\label{prop_tcxy}
We have $\TC_n(f) \leq \min \{\TC_n(X), \TC_n(Y)\}$ for all  maps $f \colon X \to Y$.
\end{proposition}
\begin{proof}
This follows from Proposition \ref{prop_inequality_1} and \cite[Proposition 3.7]{BS2}.
\end{proof}
We note that the sectional category of a map does not exceed the LS-category of the the codomain. Thus   $$\TC_n(f) \leq \min \{n \cat(X), n \cat(Y)\}.$$ 
A part of the following result generalizes \cite[Proposition 3.3]{Rud10}.

\begin{proposition}\label{prop_cat_tc}
Let  $f \colon X \to Y$ be a map. Then $$\cat(f) \leq \TC_n(f) \leq \TC_{n+1}(f) \leq \cat(f^{n+1}) \leq (n+1) \cat(f)$$ for all  $n$. 
\end{proposition}
\begin{proof}
Let $x \in X$ and $$\phi_{n+1}|_{(B, B \cap \Delta_{n+1}(X))} \colon   (B, B \cap \Delta_{n+1}(X)) \to  (Y^{n+1}, \Delta_{n+1}(Y))$$ inessential for some open subset $B \subset X^{n+1}$.   Let $A= X^n \times {x}$,  $\phi_n'$ be the restriction of $\phi_{n+1}$ on $A$, and $pr_{x} \colon Y^{n+1} \to Y^n \times f(x)$ the projection. Then $\phi_n'$ is $\phi_n$ and  the following composition  $$pr_{x} \circ \phi_n' \colon (A \cap B, A \cap B \cap (\Delta_{n}(X), x)) \to (Y^n \times f(x), (\Delta_n(Y), f(x)))$$ is inessential. This proves the second inequality. The first inequality follows from similar arguments. The third inequality follows from their definitions. The last inequality follows from \eqref{eq_prod}. 
\end{proof}
We remark that \eqref{TC_RS_MW}, Proposition \ref{prop_tcxy}  and \ref{prop_cat_tc} imply \cite[Proposition 3.8]{Sco} in particular. Moreover, growth of $\TC_n(f)$ is linear like $\TC_n(X)$. Let $X, Y$ be a non-contractible space and $f \colon X \to Y$ constant map. Then $\TC_n(f) = 1$, but $\mbox{cat} X \geq 2$ and $\mbox{cat} Y \geq 2$. So we cannot say $\mbox{cat} X$ or $\mbox{cat}(Y)$ is a lower bound for $\TC_n(f)$ in general. However we can give the following cohomological lower bound by Theorem \ref{thm:cohom_low_bnd}.

\begin{proposition}
Let $f \colon X \to Y$ and $\overline{\phi}_n$ be as in  Definition \ref{def:TC_map}. Then ${\rm nil}{\rm Im}(\overline{\phi}_n^*) \leq \TC_n(f)$ for any $n$. 
\end{proposition}
 
 The map $\Delta_2 \colon X \to X \times X$ induces a map $\cup \colon H^*(X) \otimes H^*(X) \to H^*(X)$ defined by the cup product. If $f \colon X \to Y$ is a map, then $\ker \cup|_{\mbox{Im}{(f\times f)}^*} \subseteq {\rm Im}{(f\times f)}^*$. Therefore we get the following inequalities.
 \begin{corollary}
${\rm nil} \ker U|_{{\rm Im}{(f\times f)}^*} \leq  {\rm nil}{\rm Im}{(f\times f)}^*,   \TCMW(f) \leq \TC_2(f).$ 
\end{corollary}
\forget
 Therefore, $\TC_2(f)$ is a better lower bound than ${\rm nil} \ker U|_{\mbox{Im}{(f\times f)}^*}$ of \cite[Proposition 1.8]{MuWu}.
 \forgotten
 
  \begin{example}
 Let $L(p; q_1, \ldots, q_n)$ be the generalized lens space and $S^{2n+1} \xrightarrow{f} L(p; q_1, \ldots, q_n)$ be the corresponding orbit map. So $\deg(f)=p$. That is $f$ is essential. Thus $1 \leq \TC_2(f)$. On the other hand $\TC_2(f) \leq \TC(S^{2n+1}) = 1$ by  Proposition \ref{prop_tcxy} and \cite[Theorem 8]{Far}. So $\TC_2(f)=1$ and it does not depend on the degree of $f$.
 \end{example}
 \begin{example}
 Let $n \geq 3$ be odd and $\mathbb{R}P^n \xrightarrow{g} \mathbb{R}P^n/\mathbb{R}P^{n-1} \cong S^n$ the natural quotient map. So $g$ is essential which implies that  $1 \leq \TC_2(f)$. On the other hand, $\TC_2(g) \leq \TC(S^n)$ by  Proposition \ref{prop_tcxy}. Thus $\TC_2(g)=1$. Whereas the results in \cite{MuWu} cannot determine $\TCMW(g)$. However, we can have 
 \[
 1 \leq \TCMW(g) \leq \TC_2(g)=1.
 \]
%  by \cite[Proposition 2.2]{MuWu}. We note that $\TC(\mathbb{R}P^n) \geq \cat(\mathbb{R}P^n=n\geq 3)$ and $$\cat(g \times g) \leq \mbox{min}\{\cat(\mathbb{R}P^n \times \mathbb{R}P^n), \cat(S^n \times S^n)\}=2.$$ Therefore $1 \leq \TCMW(g) \leq 2$.  
 \end{example}

 \forget
 {\bf I am thinking to add some examples here to see if our notion is different from Murillo and Wu.}
 \forgotten

%======================================================

\section{Maps of Degree 1 and Lusternik--Schnirelmann category}\label{sec_lscat_deg_map}
\m In this section, we give affirmative answer of Question \ref{Rud_conj} in a broader category, other than the class of manifolds discussed in \cite{Rud17}. Please, pay attention to \remref{r:deg}. For future references, we note the following 2 known facts.

\begin{proposition}\label{p:cat=1} 
If $M$ be a closed manifold with $\cat M =1$ then $M$ is a homotopy sphere.  
\end{proposition}

\begin{proof} 
Recall that $\pi_1(X)$ is free for all CW spaces (not necessarily manifold) $X$ with $\cat X=1$ (i.e the so-called co-$H$-spaces), see~\cite[Ex.1.21]{CLOT} or \cite[Prop. 2.4.3]{Ar}. For simply connected closed manifold $M^n$ we claim that $H^i(M)=0$ for $i\neq 0,n$ (and hence $M$ is a homotopy sphere). Indeed, if $a\in H^k(M), a\neq 0$ for some $k\neq 0,n$ then there is a Poincar\'e dual class $b\in H^{n-k}(M)$ with $a\smile b\neq 0$. Hence  cup-length of $M$ is at least 2, and so $\cat M\geq 2$, a contradiction. If $M$ is not simply connected then $\pi_1(M)$ is non-trivial free group, and so $H_1(M)$ is a non-zero free abelian group, and so $H^1(M)\neq 0$. Hence, asserting as above, we see that cuplength of $M$ is at least 2. Thus, $\cat M \geq 2$, a contradiciton.
\end{proof}

%\m The following open problem was asked by Rudyak \cite{Rud99} and now it is  consider as a conjecture, cf. \cite{Dra}. 

%\begin{question}\label{Rud_conj}
%Given two closed manifolds $M$ and $N$, assume that there exists a map
%$f \colon M \to N$ of degree 1. Is it true that $\cat (M) \geq \cat  N $?
%\end{question}

\begin{prop}\label{p:deg}
Let $f \colon M\to N$ be a map of degree $\pm1$. Then $f_* \colon H_*(M) \to H_* N $ is an epimorphism and 
$f^* \colon H^* (N)  \to H^*(M)$ is a monomorphism. Furtherbore, $f_* \colon \pi_1(M)\to \pi_1 N $ is an epimorphism.
\end{prop}

\begin{proof}
For $H_*,H^*$ see ~\cite[Theorem V.2.13]{Rud98}, cf. also \cite {Dyer}. For $\pi_1$ see~\cite[Prop 5.1]{DR}.
\end{proof} 

\begin{cor}\label{c:deg}
Let $f \colon M\to N$ be a map of degree $\pm1$. If $M$ is a homotopy sphere then $N$ is.
\end{cor}

\m We recall the notion of category weight, see \cite{Rud96, Rud99, Strom}. 

\begin{df}\label{wgt}
Given a connected CW space $X$ and spectrum (cohomology theory) $ E^*$, the category weight of $u\in E^*(X)$, denoted by $\wgt u$ is defined by
\[
\wgt u \geq k \text{ if and only if }\gf^*(u)=0 \text { for all }\gf \colon A\to X \text{ with } \cat \gf >k.
\]
\end{df}

Another definition looks as follows: $\wgt u$ is the greater integer $k$ such that $p^*_{k-1}(u)=0$. Here $p_{k-1} \colon G_{k-1}(X)\to X$ is the $(k-1)$th Ganea fibration, \cite[Definition 1.51]{CLOT}. For the proof of equivalence of these two definitions, see Section 2.7 in \cite[Lemma 8.21]{CLOT}

\forget
It is worth noting that the Ganea fibration $p_{k-1} :G_{k-1}(X)\to X$ is homotopy equivalent to the fibration $p_k:P_k(X)\to X$ in~\cite[(1.4)]{Rud99}, the proof is indicated in~\cite[Example 1.61]{CLOT}.  Furthermore, let 
\[
\subset B_k\Omega X \subset B_{k+1}\Omega X\subset \cdots \subset B\Omega X\subset X
\]
be the Milnor filtration  of the classifying space $B\Omega X$ of $X$. Then the map $i_k: B_k\Omega X \subset B\Omega X\simeq $ homotopy equivalent to the fibration $p_k: P_k(X)\to X$. So, the last two maps $p_k$ and $i_k$ can also be used for the definition of category weight.
\forgotten

% This definition looks a bit differently from original Rudyak and Strom definitions, but, of course, they are the same.

\begin{rem}
 The origin of notion of category weght goes back to Fadell and Husseini~\cite{FH}. However, their definition has a disadvantage that it is not a homotopy invariant.
 \end{rem}

\begin{prop}\label{p:wgt}
Category weight has the following properties.
\begin{enumerate}
\item $\wgt u\leq \cat X\leq \dim X$ for all $u\in E^*(X)$;
\item $\wgt (f^*(u)\geq \wgt u$ for any $u\in E^*(X)$ and $f \colon Y\to X$ with $f^*(u) \neq 0)$;
\item If $E^*$ is ring spectrum then $\wgt(u\smile v)\geq \wgt u + \wgt v$ for $u,v\in E^*(X)$;
\item If $u\in H^s(K(\pi,1);R)$ then $\wgt u\geq s$.
\end{enumerate}
\end{prop}

\begin{proof}
See~\cite[Theorem 1.8]{Rud99} or~\cite[Prop. 8.22]{CLOT}
\end{proof}

\begin{df}[\cite{Rud99}]
Let $X$ and $E$ be as in \defref{wgt}. An element $u\in E^*$ is a {\em detecting element} for $X$ is $\wgt u=\cat(X)$. 
\end{df}

\begin{proposition}[{\cite[Section 4]{Rud17}}]\label{prop_rud17}
Let $f \colon M \to N$ be a map of degree $\pm 1$. If dimension of $M, N$ is less than 5, then $\cat(M) \geq \cat(N) $.
\end{proposition}

We note that if $f \colon M \to N$ is a degree one map and $\cat(N) \leq 2$, then $\cat(M) \geq \cat(N)$ in any dimension. 
  
\forget {\bf Soumen, I do not see where did you get item 1. I think, we must remove it}
%In this section we extend these class of manifolds where the Question \ref{Rud_conj} has positive answer.  

%Let $X$ be a topological space and $E$ a cohomology theory.  
\forgotten 
%\section{Some Examples}
\m The following propositions give an effect of having detecting element after taking finite product with spheres.

\begin{proposition}
Let  $g \colon M  \to 
N \times S^{n_1} \times \cdots \times S^{n_\ell}$ be a  map of degree $\pm 1$ for some 
$\ell \geq 0$.
\begin{enumerate}[\rm (i)] 
\item If $N$ possesses a detecting element in $H^*$ then 
\[
\cat(M) \geq \cat(N \times S^{n_1} \times \cdots \times S^{n_\ell}).
\] 
\item If $M$ and $N$ are stable parallelizable manifold and $N$ is a $(q-1)$-connected manifold with $q \cat(N)  = \dim N \geq 4$ 
then 
\[
\cat(M) \geq \cat(N \times S^{n_1} \times \cdots \times S^{n_\ell}).
\]
\end{enumerate}
\end{proposition}

%\begin{proposition}
%Let  $g \colon M \times S^{n_1} \times \cdots \times S^{n_k} \to 
%N \times S^{m_1} \times \cdots \times S^{m_\ell}$ be map of degree $\pm 1$ for some 
%$k, \ell \geq 0$. If $N$ possesses a detecting element then $cat(M \times S^{n_1} \times \cdots \times S^{n_k}) \geq cat(N \times S^{m_1} \times \cdots \times S^{m_\ell})$.
%\end{proposition}
\begin{proof}
For sake of simplicity, put $N^*:= N \times S^{n_1} \times \cdots \times S^{n_\ell}$. 
\begin{enumerate}[\rm (i)] 
\item If $N$ possesses a detecting element in $H^*$, then by \cite[Corollary 2.3]{Rud99} 
$N^*$ possesses a detecting element, say, $u$. Since $\deg g =\pm 1$ we conclude that $g^*: H^*(N^*)\to H^*(M)$ is a monomorphism by~\cite[Lemma 2.2]{Rud17}. Thus, by~\propref{p:wgt} we have 
\[
\cat(M) \geq \wgt g^*u\geq \wgt u=\cat(N^*).
\]
\item If $M,N$ are stable parallelizable manifold then $M,N$ is $S$-orientable with respect to the sphere spectrum $S$.  Recall that $N$ is a $(q-1)$-connected manifold with $q \cat(N)  = \dim N \geq 4$. Hence by~\cite[Theorem 3.1]{Rud99}, $N$ possesses a detecting element $v\in E^*(N) $ for some spectrum $E$. So, $N^*$ has a detecting element for $E$ by ~\cite[Lemma 2.2]{Rud99}. Since every spectrum is a module spectrum over $S$, we conclude that $g^*: E^*(N^*)\to E^*(M)$ is a monomorphism by~\cite[Theorem V.2.13]{Rud98}. Now the proof can be completed as the previous part. 
\end{enumerate}
\end{proof}

\begin{proposition}\label{prop_deg_mton2}
Let $f \colon M \to N$ be a map of degree $\pm 1$ and suppose that
\[ 
\cat(M \times S^{m_1} \times \cdots \times S^{m_k}) \geq \cat(N \times S^{n_1} \times \cdots \times S^{n_k}),\quad k\geq 1.
\] If either $N$ possesses a detecting element or $ \dim N \leq 2\cat(N) -3$ then $\cat(M) \geq \cat(N)$.
\end{proposition}

\begin{proof}
If $N$ possesses a detecting element, then by \cite[Corollary 2.3]{Rud99} we have
$\cat(N \times S^{n_1} \times \cdots \times S^{n_k}) = \cat(N)  + k$.

Also, if $\dim N \leq 2\cat(N)  +3$, then by \cite[Theorem 3.8]{Rud99}, we have
$\cat(N \times S^{n_1} \times \cdots \times S^{n_k}) = \cat(N)  + k$.

Since $\cat(X \times Y) \leq \cat(X) + \cat(Y)$ for all $X,Y$, we have 
\[
\cat(M \times S^{m_1}\times \cdots \times S^{n_k}) \leq \cat(M) + k.
\]
 This implies the conclusion.
\end{proof}

\forget { \it We note that the conclusion of Proposition \ref{prop_deg_mton2} is true
even if there is no degree one map from $M$ to $N$}. 

\m {\bf Soumen, I do not understand what do you mean in your note.}
\forgotten

\m Now, let $N =(N^{2n}, \omega)$ be a closed symplectic manofold with the closed non-degenerate symplectic 2-form $\omega$. Let $[\omega] \in H^{2}(N; \RR)$  be the de Rham cohomology class of $\omega$.

\begin{proposition}
Let $f \colon M \to N$ be a map of degree $\pm 1$ and $N$ is a symplectic manifold as above. In addition, assume that $N$ is simply connected . Then $\cat(M) \geq \cat(N)$.
\end{proposition}

\begin{proof}
By the cup-length theory for LS category, we have $\cat(N)  \geq n$ as $[\omega]^n \neq 0$. Furthermore, $\cat(N)  \leq n$ as $N$ is simply connected of dimension $2n$. Hence, $\cat(N)  = n$. Since $f^* \colon H^*(N;\RR) \to H^*(M;\RR)$ is injective by~\cite[Lemma 2.2]{Rud17}, we have $f^*([\omega])^n \neq 0 $. Therefore,  $\cat(M) \geq n$.
\end{proof}

This results can be improved as follows.

\begin{proposition}
Let $f \colon M \to N$ be a map of degree $\pm 1$ and $(N^{2n}, \omega)$ a symplectic manifold such that the de Rham cohomology class $0\neq[\omega]\in H^*(N;\RR)$ vanishes on the image of $\psi_N \colon \pi_2 (N)  \to H_2 (N) $ where $\psi$ is the Hurewicz homomorphism. Then $\cat(M) =\cat(N)=2n$.
\end{proposition}

\begin{proof}
Since $[\omega]$ vanishes on $\psi$, there exist a map $g \colon N \to K(\pi_1 (N) ; 1)$ and cohomology class $\omega_{\pi}\in H^2(K(\pi_1 (N) ; 1)$  such that $g^*(\omega_{\pi})=[\omega]$, see~\cite[Lemma 8.18]{CLOT}. Note that $\wgt (\omega_{\pi})\geq 2$, and so $\wgt[\omega]\geq 2$, by \proprefs{p:wgt}(2,4).

  Since $\omega$ is a symplectic form, $[\omega^n]\neq 0$, and so $g^*(\omega^n_{\pi})\neq 0$. Hence
\[
\cat(N) \geq \wgt (g^*(\omega^n_{\pi}))\geq n \wgt(g^*\omega_{\pi})\geq 2n.
\]
So, $\cat(N) = 2n$. Furthermore $\wgt(f^*[\omega^n])\neq 0$ because $f^*$ is monic ($\deg f=\pm 1$). Recall that $\wgt [\omega] \geq 2$. Thus, by \propref{p:wgt} 
\[
\cat(M) \geq \wgt f^*[\omega^n]\geq \wgt [\omega^n]=n\wgt[\omega]\geq 2n.
\]
\end{proof}

\begin{proposition}
Let $f \colon M \to N$ be a map of degree $\pm 1$. Suppose that $\dim M =5=\dim N$, and $\cat(N) = 5$. Then $\cat(M) \geq \cat(N)$.
\end{proposition}

\begin{proof}
By \cite[Corollary 3.3 (ii)]{Rud99}, there is a detecting element for $N$. Now the result follows from \cite[Theorem 1.8 (iii)]{Rud99}. 
\end{proof}

%I'll write a proof as we discussed. Use paper of Rudyak \cite{Rud99} to find a detecting element, Corollary 3.3.  

\begin{proposition}
Let $f \colon M \to N$ be a map of degree $\pm 1$. Assume that $N$ a simply connected manifold, and dimension of $M, N$ is equal to 5 or 6. Then $\cat(M) \geq \cat(N)$.
\end{proposition}

\begin{proof}
First, suppose that $\dim N=5$. Then $\cat(N) \leq 2$ by \cite[Theorem 1.50]{CLOT}. The case $\cat(N)=1$ is obvious. Furthermore, if $\cat(N)=2$, then $\cat(M) \geq 2$ . Indeed, by way of contradiction, suppose that $\cat(M)=1$. Then, by \propref{p:cat=1}, $M$ is homotopy equivalent to the sphere $S^5$. Now,  the condition $\deg f=\pm 1$ implies that $N$ is a homotopy sphere by \corref{c:deg}, and thus $\cat(N)=1$. This is a contradiction.

\smallskip
Let $\dim N=6$. Then $\cat(N)  \leq 3$. If $\cat(N)  \leq 2$, then argument is
similar to the previous paragraph. Now suppose that $\cat(N) = 3$, and so $2 \cat(N)  =6 \geq 4$.
 It follow from \cite[Corollary 3.1]{Rud99} that $N$ has a detecting element $u$. Moreover, we know that $f^*(u)\neq 0$ because $\deg f=\pm 1$. This implies 
 that $\wgt f^*(u)\geq \wgt u=6$ by \propref{p:wgt}, and thus $\cat(M) = 6$. 
\end{proof}

\begin{proposition}
Let $f \colon M \to L_p^{2n+1}$ be a map of degree $\pm 1$ where $L_p^{2n+1}$ is a lens 
space. Then $\cat(M) \geq \cat(L_p^{2n+1}) $.
\end{proposition}

\begin{proof}
We know that  
\[
H^*(L_p^{2n+1};\ZZ/p)=\ZZ/p[u,\gb u|u^2=0, (\gb u)^{n+1}=0]
\]
where $\gb$ is the Bockstein homomorphism $\mod p$. Furthermore, $\wgt u=2$,~\cite{FH, Rud99}. So, $\wgt(u\smile(\gb u)^n)\geq 2n+1$ by \propref{p:wgt}(3). Hence, $\wgt(u\smile(\gb u)^n)=2n+1$ by \propref{p:wgt}(1). Therefore, $f^*(u\smile(\gb u)^n))\neq 0$ by \cite[Lemma 2.2]{Rud17}. 

Thus, $\cat(L_p^{2n+1}) \geq \wgt (f^*(u\smile(\gb u)^n)=2n+1$ by \propref{p:wgt}.
\end{proof}

\begin{proposition}
Let $f \colon M \to N$ be a map of degree $\pm 1$ and $N$ a sphere bundle over sphere with a cross-section. Then $\cat(M) \geq \cat(N)$.
\end{proposition}
\begin{proof}
Let $q \colon N \to S^n$ be a sphere bundle with fiber $S^m$. So, by (3.3) in \cite{JaWh}, $N$ has a cell complex structure given by $N = e^0 \cup e^m \cup e^n \cup e^{m+n}$. Since $N$ has a cross-section, $e^0 \cup e^m \cup e^n$ is homeomorphic to $S^m \vee S^n$, see discussion after (3.3) in \cite{JaWh}. Thus $\cat(N) \leq 2$. So by the remark after Proposition \ref{prop_rud17} the conclusion follows. 
\end{proof}

\m We recall the definition of toric manifold briefly following \cite{DJ}. 
A toric manifold is a  $2n$-dimensional smooth manifold equipped with an $n$-dimensional torus action such that the action is locally resemble the $T^n$-action on $\CC^n$ and the orbit space has the structure of an $n$-dimensional simple polytope.    
\begin{proposition}\label{prop_toric_mfds}
Let $f \colon M \to N$ be a map of degree $\pm 1$ and $N$ a toric manifold. Then $\cat(M) \geq \cat(N)$.
\end{proposition}

\begin{proof}
Let $\mathfrak{q} \colon N \to P$ be the orbit map where $\dim(N)=2n$ and $P$ is an $n$-dimensional polytope. Let $v$ be a vertex of $P$. So there is unique $n$ many codimension one faces $F_1, \ldots, F_n$ of $P$ such that $v = F_1 \cap \cdots \cap F_n$.  From the discussion just above \cite[Proposition 3.10]{DJ}, we get that the codimension-2 submanifolds $\mathfrak{q}^{-1}(F_1), \ldots, \mathfrak{q}^{-1}(F_n)$ intersects transversely and $\mathfrak{q}^{-1}(v) = \mathfrak{q}^{-1}(F_1) \cap \cdots \cap \mathfrak{q}^{-1}(F_n)$ is a point. Let $x_i$ be the Poincar{\'e} dual of $\mathfrak{q}^{-1}(F_i)$.  So $x_1, x_2, \ldots, x_n \in H^2(N; \ZZ)$ such that $x_1 \smile \cdots \smile x_n $ is nonzero in $H^*(N; \mathbb{Z})$.   This implies that cup-length of $N$ is at least $n$. Also $N$ is simply connected by \cite[Cor. 3.9]{DJ}.  Now, the proposition follows from \cite[Proposition 3.3]{Rud17}.  
\end{proof}

\forget 
{\it Soumen, the proof of 4.8 is not correct (or, at least, incomlete). It can happen that $\{x_1, x_2, x_3\}$ generate $H^4(N)$ but no of 2-products is equal to $H^4(N) $. }

{\bf Soumen, what is a good reference on toric manifolds. Wikipedia is not good.}

\m By arguments similar to the proof of Proposition \ref{prop_toric_mfds} and
\cite[Lemma 4.2]{BS1}, we get the following.  
\forgotten 
 
This result can be extended to the class of locally standard torus manifolds \cite[Section 4]{MaPa} where the orbit space may not be a simple polytope always.
\begin{proposition}
Let $f \colon M \to N$ be a map of degree $\pm 1$ and $N$ is a $2n$-dimensional
locally standard torus manifold.
If the orbit space $N/T^n$ is simply connected and contains a boundary of a simple polytope, then $\cat(M) \geq \cat(N)$.
\end{proposition}
\begin{proof}
In the proof of \cite[Lemma 4.2]{BS1}, the authors show that there are cohomology classes  $x_1, x_2, \ldots, x_n \in H^2(N; \ZZ)$ such that $x_1 \smile \cdots \smile x_n $ is nonzero in $H^*(N; \mathbb{Z})$. Also, $\cat(N)=n$ by \cite[Theorem 4.3]{BS1}.  Then the result  follows from \cite[Proposition 3.3]{Rud17}.
\end{proof}

\begin{rem} Dranishnikov and Scott~\cite{DS} proved the inequality $\cat M \geq \cat N$ in case of the existence of a {\it normal map} $M\to N$ of degree 1, provided that $\dim N \leq 2r\cat N-3$. Here $M$ and $N$ are closed manifolds and the $N$ is $(r-1)$-connected, $r\geq 1$.    
\end{rem} 

\section{Maps of Degree 1 and Topological Complexity}\label{sec_topcom_deg_map}
% Farber \cite{Far} introduced topological complexity in connection with robot motion planning. 
In this section, we compare the topological complexity of two orientable manifolds when there is a degree one map between them.
We recall that $\TC_2(X)$ of Definition \ref{def_higher_tcn} is Farber's $\TC(X)$. 
\forget 
Let $X$ be a path connected space and $X^I$ be the free paths in $X$.

 $\alpha \colon [0, 1] \to X$. We recall the definition of sectional
category. The {\em sectional category} of a fibration $f \colon X \to Y$,
denoted by $\secat(f)$, is the least integer $m$ such that $Y$ can be 
covered by $m$ open sets $U_0, U_1,\ldots, U_m$ for each of which there exists
a homotopy section to $f$, i.e. there is a map $s_j \colon  U_j \to X$ such
that $f \circ s_j \simeq \iota_{U_j} $. 
If  no such integer exists then we put $\secat(f) = \infty$.

\m {\it Maybe to write it in the preliminary part. Also, see what I wrote in the beginning of the paper.}

\m Now consider the free path fibration as in \eqref{pi}
\[
\pi=\pi_X: X^I \to X\ts X, \quad \pi(\ga)=(\ga(0),\ga(1)), \, \ga:I\to X.
\]

Then {\em topological complexity}, denoted by $\TC(X)$,
is defined by $\TC(X)= \secat(\pi)$.
\forgotten 
 The following proposition is proved in \cite{LuMa}.

\begin{proposition}\label{prop_lub_marz}
Let $\Delta X$ be the diagonal space of the product $X \times X$. 
Then $\TC(X) = \relcat(X \times X , \Delta X)$ if $\TC(X)$ is finite. 
\end{proposition} 
 
Analogously to the Question \ref{Rud_conj}, one may ask the following.

\begin{question}\label{ques_top_comp_deg_1}
Let $n \geq 2$ and $M, N$  two closed orientable manifolds such that there exists a map
$f \colon M \to N$ of degree $\pm 1$. Is it true that $\TC_n(M) \geq \TC_n(N)$?
\end{question}

If $\dim M, N=1 $  then it is clear. In the remaining,
we show that this question has affirmative answer when dimension of 
$M,N$ are 2 and 3. Note that if the degree of the map in
Question \ref{ques_top_comp_deg_1} is more than one then the question
has negative answer, for example one can take degree $k$ map from 3-sphere
to lens space. We also analyse this question on several types of oriented
manifolds.  

\forget 
\m {\it Below  I modify the previous text on zero-divisors, etc. I copeid our previous text by red letters.}
\forgotten

\begin{df}
Given a space $X$, a commutative ring $R$, and a positive integer $n$, we say that $u\in H^*(X^n;R)$ is a {\em zero-divisor class of grade $n$} if $d_n^*u=0$. Here $d_n \colon X \to X^n$ is the diagonal map. A {\em zero-divisors-cup-length of grade $n$} of $H^*(X; R)$, denoted by $\zcl_n(X)=\zcl_n^R(X)$ is the maximal $k$ such that $u_1\smile \cdots \smile u_k\neq 0$ provided each $u_i$ is a zero-divisor class of grade $n$.
\end{df}

We note that the cup product in $H^*(X;R)$ yields a homomorphiam
\[
\CD
 (H^*(X;R))^{\otimes n} @>>> H^*(X^n;R) @>d_n^*>> H^*(X;R).
\endCD
\]
We also denote the above composition by $d_n^*$. 
\begin{thm} \label{t:cup}
For any commutative ring $R$ we have the inequality $\TC_n(X) \geq \zcl_n^R(X)$.
\end{thm}

\begin{proof}
This theorem follows from~\cite[Theorem 4]{sv} if we replace the fibration $p \colon E \to B$ in the \cite[Theorem 4]{sv} by the fibration $\pi_n \colon X^I \to X^n$ in \eqref{pin}.
\end{proof}

\begin{prop}\label{p:injectivity of zcl}
Let  $f \colon M \to N$ is a map of degree $\pm 1$. Then for any commutative ring $R$ we have $\zcl_n(M) \geq \zcl_n(N)$. In particular, if $\TC_n(N)=\zcl_n(N)$ then $\TC_n(M)\geq \TC_n(N)$.
\end{prop}

\begin{proof}
Let $\zcl_n(N)=k$. Consider elements $u_1, \ldots, u_k\in H^*(N^n;R)$ with $d^*_n(u_i) = 0$ for $i=1, \ldots, k$ and such that $u_1\smile \cdots \smile u_k\neq 0$. Now we have $f^*(u_1)\smile \cdots \smile f^*(u_k)\neq 0$ by ~\cite[Lemma 2.2]{Rud17}. So the result follows. 
\end{proof}

\begin{prop}
Let $M\to S^{k_1} \ts S^{k_2} \ts \cdots \ts S^{k_m}, k_i>0$ be a map degree $\pm 1$. Then
\[
\TC_n (M) \geq \TC_n(S^{k_1}\ts S^{k_2}\ts \cdots \ts S^{k_m}).
\]
\end{prop}

\begin{proof}
Put $N=S^{k_1}\ts S^{k_2}\ts \cdots \ts S^{k_m}$. Then, by \cite[Cor. 3.12]{BGRT}, we have $\TC_n(N) = m(n-1)+ \ell $ where $\ell$ is the number of even dimensional spheres. (Note that the invariant $\cl(X,n)$ in~\cite{BGRT} is exactly our invariant $\zcl_n(X)=\zcl_n^{\ZZ}(X)$). By~\cite[Theorem 3.10]{BGRT} we have  $\zcl_n(S^{2k+1})=n-1$ and $\zcl_n(S^{2k})=n$, and, moreover, the same theorem yields the inequality
$
\zcl_n(N)\geq m(n-1)+ \ell =\TC_n(N)
$
by induction. So, $\TC_n(N) = \zcl_n(N)$ by \theoref{t:cup}. Thus, by \propref{p:injectivity of zcl} we have $\TC_n(M)\geq\zcl_n(M)\geq \zcl_n (N)=\TC_n(N)$.
\end{proof}

\begin{proposition}
Let $M$ be a stably parallelizable 14-dimensional. Consider the  exceptional Lie group $G_2$, see e.g. \cite{Adams}. If $f \colon M \to G_2$ is a map of degree $\pm 1$ then
$\TC_n(M) \geq \TC_n(G_2)$.  
\end{proposition}
\begin{proof}

We have  $\cat(M) \geq \cat(G_2)$ by~\cite[Prop. 5.4]{Rud17}. Furthermore,
\cite[Theorem 3.5]{BGRT} says that $\TC_n(G_2) = (\cat(G_2))^{n-1}$. So proposition
follows from the inequality $(\cat(X))^{n-1} \leq \TC_n(X)$ for all $X$,~\cite[Cor. 3.3]{BGRT}.
\end{proof}

\begin{proposition}
Let $M$ be a closed connected smooth manifold and $f \colon M \to SO(n)$
is a degree one map, then $\TC(M) \geq \TC(SO(n))$ if $n \leq 9$.  
\end{proposition}
\begin{proof}
We have $\cat(M) \geq \cat(SO(n))$ if $n \leq 9$ by~\cite[Theorem 5.5]{Rud17}. Furthermore, \cite[Theorem 3.5]{BGRT} says that $\TC(G) = \cat(G)$ for a connected Lie group. So proposition follows from the inequality $(\cat(X))^{n-1} \leq \TC_n(X)$.
\end{proof}

\begin{proposition}
If $f \colon M \to N$ is a map of degree $\pm 1$ and~$\TC(N)=2,$ then $\TC(M) \geq 2$. 
\end{proposition}
\begin{proof}
If $\TC(M) = 1$ then $M$ must be   an odd-dimensional homotopy sphere, see \cite{GLO}. Thus, $N$ must  be an odd-dimensional homotopy sphere by \propref{c:deg}.
\end{proof}

\begin{proposition}
If $f \colon M \to N$ is a map of degree $\pm 1$ and~$\dim N=2,$ then $\TC(M) \geq \TC(N)$. 
\end{proposition}
\begin{proof}
Since $N$ is an oriented 2-dimensional manifold, it satisfies the hypothesis of Proposition \ref{p:injectivity of zcl}. So we are done. 
\end{proof}

\begin{thm}\label{t:h_sphere}
Let $M$ be an odd-dimensional oriented manifold which is not a rational sphere. If $f \colon M \to N$ has degree $\pm 1$ and $\TC(N) =3$, then $\TC(M) \geq \TC(N)$. 
\end{thm}

\begin{proof}
Since $M$ is not a rational sphere, then there is a non-zero cohomology class $\alpha \in H^i(M; \RR)$ for some $0 < i < \dim M$. Then there is a Poincar{\'e} dual, denoted by $\alpha'$,  of $\alpha$ such that $\alpha \smile \alpha'\neq 0$ in $H^{\dim M}(M; \RR)$. Let $a := 1 \otimes \alpha - \alpha \otimes 1$ and  $b := 1 \otimes \alpha' - \alpha' \otimes 1$.  So $a, b \in \zcl_2(M)$. 

If $\dim \ga$ is even then
\[
a\smile a =-2\ga\otimes \ga \neq 0.
\]
Hence for $M\ts M$ we have
\[
a\smile a \smile b=(-2\ga\otimes \ga)\otimes(\ga'\otimes 1-1\otimes \ga')=(-2\ga\smile \ga')\otimes \ga)+(2\ga\otimes\ga)\smile \ga';
\]
here the cup-product $\smile$ in the right-hand side of the last equality appears as $M$, not $M\ts M$. Since $a\smile a \smile b\neq 0$ is $\zcl_2(M)$, we conclude that $\TC(M) \geq 3=\TC(N)$.  

If $\dim a$ is odd then $\dim b$ is even, and the same argument shows that $\TC(M) \geq 3=\TC(N)$.  
\end{proof}

\begin{thm}\label{t:h_betino}
 Let $M^{2n}$ be a closed orientable manifold. Assume that there exists $k, 0<k<n$ such that  $2k$-th Betti number of $M$  is non-zero. Then $\TC(M) \geq 4$. In particular, if $f \colon M \to N$ has degree $\pm 1$ and $\TC(N) \leq 4$, then $\TC(M) \geq \TC(N)$. 
\end{thm}

\begin{proof}
There is a non-zero cohomology class $\alpha \in H^{2k}(M; \RR)$ for some $0 < 2k < \dim M$. Then there is a Poincar{\'e} dual, denoted by $\alpha' \in H^{2n-2k}(M; \RR)$,  of $\alpha$ such that $\alpha \smile \alpha'\neq 0$ in $H^{2n}(M; \RR)$. Let $a := 1 \otimes \alpha - \alpha \otimes 1$ and  $b := 1 \otimes \alpha' - \alpha' \otimes 1$ in $H^*(M\ts M;\RR)$. Then $a\smile a=\alpha\otimes \alpha$, and similarly for $b$. So $a, b \in \zcl_2(M)$ and $a\smile a \smile b \smile b\neq 0$ in $\zcl(M)$. Thus $\TC(M) \geq 4$.  
\end{proof}

\begin{proposition}
Let $M, N$ be two 3-dimensional oriented manifolds. Suppose that $\TC(N) =3$. If $f \colon M \to N$ has degree $\pm 1$, then $\TC(M) \geq 3$. 
\end{proposition}

\begin{proof}
This follows from \propref{t:h_sphere}, but give also an alternative proof. By way of contradiction, suppose that $\TC(M)\leq 2$, Than $\cat M\leq 2$. If $\cat M=1$ then $M$ is a homotopy sphere, and hence $N$ is a homotoopy sphere by \propref{p:cat=1}. This contradicts the condition $\TC(N)=3$. If $\cat M=2$ then $M$ has a free fundamental group, ~\cite{DKR}. Hence $M$ is a connected sum of copies of $S^1\times S^2$ (because $M$ is orientable), this is classical on 3-manifolds, \cite{H}. Since $\TC(S^1\ts S^2)=3$ by~\cite[Cor. 3.12]{BGRT}, we conclude that $\TC(M)=3$ by ~\cite{DrSa}. 
.
\end{proof}

\forget

By way of contradiction, suppose that $\TC(M)\leq 2$, Than $\cat M\leq 2$. If $\cat M=1$ then $M$ is a homotopy sphere, and hence $N$ is a homotoopy sphere by \propref{p:cat=1}. This contradicts the condition $\TC(N)=3$. Furthermore, if $\cat N=2$ then $\pi_1(M)$ is a free group by~\cite{DKR}. Hence, $M$ is not a rational homotopy sphere. This contradicts \theoref{t:h_sphere}.
\forgotten

\begin{proposition}
Let $f \colon M \to N$ be a map of degree $\pm 1$ and $(N, \omega)$ a simply connected symplectic manifold. Then $\TC(M) \geq \TC(N)$.
\end{proposition}
\begin{proof}
This follows from Proposition \ref{p:deg} and \ref{p:injectivity of zcl}.  
\end{proof}

{\bf Acknowledgement:}   The second author thanks University of Florida for supporting his visit, `International office IIT Madras' and `Science and Engineering Research Board India' for research grants. The authors thank John Oprea, Jamie Scott and Petar Pavesic for some helpful discussion.  

%%%%%%%%%%%%%%%%%%%%%%%%%%%%%%%%%%%%%%%%%%%%%%%%%%%%

\renewcommand{\refname}{References}

%\begin{thebibliography}{ALR}
%\bibliographystyle{alpha}
%\bibliographystyle{abbrv}
%\bibliography{bibliography.bib}

\begin{thebibliography}{CLOT03}

\bibitem[Ada96]{Adams}
John Frank Adams.
\newblock Lectures on exceptional Lie groups.
With a foreword by J. Peter May. Edited by Zafer Mahmud and Mamoru Mimura. {\em Chicago Lectures in Mathematics.} University of Chicago Press, Chicago, IL, 1996. xiv+122 pp.

\bibitem[Ark11]{Ar}
Martin Arkowitz.
\newblock Introduction to homotopy theory.
{\em Universitext.}
\newblock Springer, New York, 2011. 

%\bibitem[ASS04]{ASS}  Martin Arkowitz, Donald Stanley, and Jeffrey Strom.
%\newblock The cone length and category of maps: pushouts, products and fibrations.
%\newblock {\em Bull. Belg. Math. Soc. Simon Stevin} 11 (2004), no. 4, 517--545.

\bibitem[BGRT14]{BGRT}
Ibai Basabe, J\'{e}sus Gonz\'{a}lez, Yuli B. Rudyak and Dai Tamaki.
\newblock Higher topological complexity and its symmetrization.
\newblock {\em Algebraic \& Geometric Topology}, 14 (2014) 2103--2124.

\bibitem[BS15]{BS1}
Marzieh Bayeh and Soumen Sarkar.
\newblock Some aspects of equivariant {LS}-category.
\newblock {\em Topology Appl.}, 196:133--154, 2015.

\bibitem[BS18]{BS2}
Marzieh Bayeh and Soumen Sarkar.
\newblock Higher equivariant and invariant topological complexity.
\newblock {\em  J. Homotopy Relat. Struct.}, to appear.   

\bibitem[CLOT03]{CLOT}
Octav Cornea, Gregory Lupton, John Oprea, and Daniel Tanr\'{e}.
\newblock {\em Lusternik-{S}chnirelmann category}, volume 103 of {\em
  Mathematical Surveys and Monographs}.
\newblock American Mathematical Society, Providence, RI, 2003.

\bibitem[DJ91]{DJ}
Michael  Davis and Tadeusz Januszkiewicz.
\newblock Convex polytopes, {C}oxeter orbifolds and torus actions.
\newblock {\em Duke Math. J.}, 62(2):417--451, 1991.

\bibitem[DKR08]{DKR}
Alexander Dranishnikov, Mikhael Katz and Yuli Rudyak. 
\newblock  Small values of the Lusternik-Schnirelman category for manifolds.
\newblock {\em Geom. Topol.} 12, no. 3, 1711--1727, 2008

\bibitem[DR09]{DR}
Alexander Dranishnikov and Yuli Rudyak. 
\newblock On the Berstein-Svarc theorem in dimension 2. 
\newblock {\em Math. Proc. Cambridge Philos. Soc.} 146, no. 2, 407--413, 2009. 

\bibitem[Dra15]{Dra}
Alexander~N. Dranishnikov.
\newblock The {LS} category of the product of lens spaces.
\newblock {\em Algebr. Geom. Topol.}, 15(5):2985--3010, 2015.

\bibitem[DS19]{DrSa}
Alexander~N. Dranishnikov and Rustam Sadykov.
\newblock On the LS-category and topological complexity of a connected sum.
\newblock {\em Proc. Amer. Math. Soc.} 147 (2019), no. 5, 2235--2244.

\bibitem[DS19]{DS}
Alexander Dranishnikov and Jamie Scott. 
\newblock Surgery Approach to Rudyak's Conjecture. 
\newblock {\em Available at} arXiv:2008.06002  

\bibitem[Dyer69]{Dyer}
Eldon Dyer. 
\newblock Cohomology Theories
\newblock {\em Mathematics Lecture Note Series,}
\newblock W.A. Benjamin, Inc., New York, Amsterdam (1969).

\bibitem[FH92]{FH}
Edward Fadell and Sufian husseini 
\newblock Category weight and Steenrod operations. Papers in honor of José Adem  
\newblock {\em Bol. Soc. Mat. Mexicana} (2) 37 (1992), no. 1-2, 151–161.

\bibitem[Far03]{Far}
Michael Farber.
\newblock Topological complexity of motion planning.
\newblock {\em Discrete Comput. Geom.}, 29(2):211--221, 2003.

\bibitem[Far08]{Far2}
Michael Farber.
\newblock Invitation to topological robotics.
\newblock Zurich Lectures in Advanced Mathematics. {\em European Mathematical Society (EMS)}, Zurich, 2008. .

%\bibitem[Gan67]{Gan1} 
%Tudor Ganea. 
%\newblock Lusternik-Schnirelmann category and strong category. 
%\newblock {\em Illinois J. Math.} 11, 417--427, 1967. 

\bibitem[Gan71]{Gan2} Tudor Ganea. 
\newblock Some problems on numerical homotopy invariants. 
\newblock Symposium on Algebraic Topology (Battelle Seattle Res. Center, Seattle Wash., 1971), pp. 23--30. 
\newblock{\em Lecture Notes in Math}., Vol. 249, Springer, Berlin, 1971.

\bibitem[GLO13]{GLO} 
Mark Grant, Gregory Lupton, and John OpreaOprea
\newblock Spaces of topological complexity one. 
\newblock{\em Homology Homotopy Appl.} 15 (2013), no. 2, pp. 73--81.

%\bibitem[Hat02]{Hat} 
%Allen Hatcher. 
%\newblock Algebraic topology. {\rm Cambridge University Press}, Cambridge, 2002.

\bibitem[H76]{H}
John Hempel.
\newblock {\em 3-Manifolds.} 
\newblock Ann. of Math. Studies 86, Princeton Univ. Press, Princeton, New Jersey 1976.

\bibitem[Iwa98]{Iwa98}
Norio Iwase. 
\newblock Ganea's conjecture on Lusternik-Schnirelmann category. 
\newblock {\em Bull. London Math. Soc.}  30 (1998), no. 6, 623--634. 

\bibitem[Jam78]{J} 
Ian James
\newblock On category, in the sense of Lusternik-Schnirelmann.
\newblock{\it Topology} 17 (1978), no. 4, 331--348.

\bibitem[JW54]{JaWh} I. M. James and J. H. C.  Whitehead.
\newblock The homotopy theory of sphere bundles over spheres. I.
\newblock {\em Proc. London Math. Soc.} (3) 4 (1954), 196--218.

\bibitem[LM15]{LuMa}
Wojciech Lubawski and Wac\l\.aw Marzantowicz.
\newblock Invariant topological complexity.
\newblock {\em Bull. Lond. Math. Soc.}, 47(1):101--117, 2015.

\bibitem[LS13]{LuSc}
Gregory Lupton and J\'{e}r\^{o}me Scherer.
\newblock Topological complexity of {$H$}-spaces.
\newblock {\em Proc. Amer. Math. Soc.}, 141(5):1827--1838, 2013.

\bibitem[MP06]{MaPa} Mikiya  Masuda and Taras Panov. 
\newblock Taras On the cohomology of torus manifolds.
\newblock{\em Osaka J. Math.} 43 (2006), no. 3, 711--746. 

\bibitem[MW19]{MuWu}
Aniceto Murillo and Jie Wu
\newblock Topological complexity of the work map.
\newblock{\em Journal of Topology and Analysis}, to appear.  

\bibitem[OR02]{OpRu}
John Oprea and Yuli Rudyak.
\newblock Detecting elements and {L}usternik-{S}chnirelmann category of
  3-manifolds.
\newblock In {\em Lusternik-{S}chnirelmann category and related topics  ({S}outh
  {H}adley, {MA}, 2001)}, volume 316 of {\em Contemp. Math.}, pages 181--191.
  Amer. Math. Soc., Providence, RI, 2002.

\bibitem[Pav19]{Pav}
Petar Pave\v{s}i\'{c}.
\newblock Topological complexity of a map.
\newblock {\em Homology Homotopy Appl.}, 21(2):107--130, 2019.

%\bibitem[Pav20]{Pav2}
%Petar Pave\v{s}i\'{c}.
%\newblock Manifolds with small topological complexity.
%\newblock {\em Available on arXiv:2011.13754}.

\bibitem[RD18]{RaDe} 
Youssef Rami and Younes Derfoufi.
\newblock A variant of the topological comlexity of a map.
\newblock {\em Available at} arXiv:1809.10174


%\bibitem[RO99]{RuOp}
%Yuli~B. Rudyak and John Oprea.
%\newblock On the {L}usternik-{S}chnirelmann category of symplectic manifolds
 % and the {A}rnold conjecture.
%\newblock {\em Math. Z.}, 230(4):673--678, 1999.

\bibitem[Rud96]{Rud96} 
Yuli~B. Rudyak.
\newblock Some remarks on category weight.
{\em Preprint}, University of Heidelberg, 1996.

\bibitem[Rud98]{Rud98} 
Yuli~B. Rudyak.
\newblock On Thom spectra, orientability, and cobordism. With a foreword by Haynes Miller. {\em Springer Monographs in Mathematics.} Springer-Verlag, Berlin, 1998

\bibitem[Rud99]{Rud99}
Yuli~B. Rudyak.
\newblock On category weight and its applications.
\newblock {\em Topology}, 38, no. 1, 37--55, 1999.

\bibitem[Rud10]{Rud10}
Yuli~B. Rudyak.
\newblock On higher analogs of topological complexity.
\newblock {\em Topology Appl.}, 157,  no. 5, 916--920, 2010. Errata: Topology Appl. 157  no. 6, 1118, 2010.

\bibitem[Rud17]{Rud17}
Yuli~B. Rudyak.
\newblock Maps of degree 1 and {L}usternik-{S}chnirelmann category.
\newblock {\em Topology Appl.}, 221, 225--230, 2017.

\bibitem[Sch66]{sv}
Albert ~S. Schwarz ({\v{S}}varc).
\newblock The genus of a fiber space.
\newblock {\em American Mathematical Society Translations}, 1966, 55, 49--140.

\bibitem[Sco20]{Sco}
Jamie Scott.
\newblock On the Topological Complexity of Maps.
\newblock  arXiv:2011.10646.

\bibitem[Sta02]{Sta}
Donald Stanley. 
\newblock On the Lusternik-Schnirelmann category of maps.
\newblock {\em Canad. J. Math.} 54 (2002), no. 3, 608--633.

\bibitem[Str97]{Strom}
Jeffrey Strom.
\newblock Essential category weight. 
\newblock {\em Preprint,} University of Wisconsin-Madison, 1997.

\end{thebibliography}
%\begin{thebibliography}{DJ}

\vspace{1cm}

\end{document}